\documentclass{amsproc}
\usepackage{amsmath}
\usepackage{amsthm}
\usepackage{amssymb}

\newtheorem{theorem}{Theorem}[section]
\newtheorem{lemma}[theorem]{Lemma}
\newtheorem{pr}{Proposition}[section]

\theoremstyle{definition}
\newtheorem{definition}[theorem]{Definition}

\theoremstyle{remark}
\newtheorem{remark}[theorem]{Remark}

\numberwithin{equation}{section}

\newcommand{\norm}[1]{\Vert #1 \Vert}  
\newcommand{\abs}[1]{\lvert#1\rvert}
\newcommand{\ov}{\overline}
\newcommand{\gr}{\mathrm{grad}}
\newcommand{\ii}{\mathrm{i}}
\newcommand{\Ker}{\mathrm{Ker}\,}

\newcommand{\dif}{\mathrm{d}}

\newcommand{\di}{\mathrm{div}}
\newcommand{\tr}{\mathrm{trace}}
\newcommand{\CC}{\mathbb{C}}
\newcommand{\Cc}{\mathfrak{C}}

\newcommand{\HH}{\mathcal{H}}
\newcommand{\VV}{\mathcal{V}}

\newcommand{\RR}{\mathbb{R}}   
\newcommand{\ZZ}{\mathbb{Z}} 
\newcommand{\NN}{\mathbb{N}}
\newcommand{\Ss}{\mathbb{S}}

\begin{document}

\title{A note on higher-charge configurations for the Faddeev-Hopf model}

\author{Radu Slobodeanu}

\address{Department of Theoretical Physics and Mathematics, University of Bucharest,
P.O. Box Mg-11, RO--077125 Bucharest-M\u agurele, Romania.}

\email{radualexandru.slobodeanu@g.unibuc.ro}

\thanks{The author is grateful to Professor Tudor Ra\c tiu and the Department of Mathematics at \emph{Ecole Polytechnique F\'ed\'erale de Lausanne} for hospitality during the preparation of the present paper. This research was supported by \emph{PN II Idei Grant, CNCSIS, code 1193}.}




\begin{abstract}
\noindent We identify higher-charge configurations that satisfy Euler-Lagrange equations for the (strong coupling limit of) Faddeev-Hopf model, by means of adequate changes of the domain metric and a reduction technique based on $\alpha$-Hopf construction. In the last case it is proved that the solutions are local minima for the reduced $\sigma_2$-energy and we identify among them those who are global minima for the unreduced energy. 
\end{abstract}

\maketitle

\section{Introduction}
\subsection{Motivation from hadrons physics} 
Skyrme model, stated in the early 60ties \cite{sky}, as well as Faddeev model proposed about ten years later \cite{fad} are attempts to apply the soliton mechanism for particle-like excitations. Explicitly, for the first case the idea was to model baryons as smooth stable finite energy solutions (\textit{solitons}) of a modified nonlinear $\sigma$-model with pion fields, while in the second case, it was suggested that gluon flux tubes in hadrons are modelled by solitons in a similar $\sigma$-model, the main difference being that the former ones were point-like (localized around a point) while the latter are knotted (localized around a loop).

To be more specific, let us present the original version of Faddeev's model, also known as Faddeev-Hopf or Faddeev-Skyrme model. The fields in this model are maps $\overrightarrow{n}$ from $\RR^3$ to the two-sphere $\mathbb{S}^2$, asymptotically constant at infinity. In the static limit the energy of the system is:
$$
\mathcal{E}_{\texttt{Faddeev}} (\overrightarrow{n})=\frac{1}{2} \int_{\mathbb{R}^3} \left\{\norm{d \overrightarrow{n}}^2
+ K \langle d \overrightarrow{n} \times d \overrightarrow{n} \ ,  \overrightarrow{n}\rangle ^2 \right\}\dif^3 x.
$$
where $K$ is a positive coupling constant. The second (fourth power)  term give the possibility of field configurations that are stable under a spatial rescaling. Moreover the following topological lower bound holds: \ 
$\mathcal{E}_{\texttt{Faddeev}}(\overrightarrow{n}) \geq c \cdot \abs{Q(\overrightarrow{n})}^{3/4}$, where $c \neq 0$ is a numerical constant and $Q(\overrightarrow{n}) \in \pi_3(\mathbb{S}^2) \cong \mathbb{Z}$ denotes the \textit{Hopf invariant} ("charge") of $\overrightarrow{n}$ seen as map on $\Ss^3$. The \textit{position} of a field configuration is defined as the preimage of the point $(0, 0, -1)$ (antipodal to the vacuum $\overrightarrow{n}_{\infty}$), so it forms a closed loop.

Let us mention that solutions $\Ss^4 \to \Ss^2$ for the strongly coupled model play also a role, in the quantized version of the theory. For more details on physical models that allow topological solitons, see \cite{mant}. 

\subsection{Differential geometric background} 
The static Hamiltonian of both Skyrme and Faddeev-Hopf models is interpreted as $\sigma_{1,2}$-\textit{energy} of mappings $\varphi: (M,g) \to (N,h)$ between Riemannian manifolds (see \cite{slobi} following \cite{man}):
\begin{equation}
\mathcal{E}_{\sigma_{1,2}}(\varphi) =
\mathcal{E}_{\sigma_{1}}(\varphi) + K \cdot \mathcal{E}_{\sigma_{2}}(\varphi) =
\frac{1}{2} \int_M \left[ \abs{\dif \varphi}^2 + 
K \cdot \abs{\wedge ^2 \dif \varphi}^2\right]  \nu _g .
\end{equation} 
The first term is the standard Dirichlet (quadratic) energy of $\varphi$ and the second (quartic) term is the $\sigma_2$-energy introduced in 1964 by Eells and Sampson \cite{els}. Their critical points are the well-known \textit{harmonic maps} and the less studied $\sigma_2$-\textit{critical} maps, respectively. We shall refer to $\mathcal{E}_{\sigma_{2}}$ as \textit{strongly coupled} energy. The critical points for the \textit{full} energy, or $\sigma_{1,2}$-\textit{critical} maps, are characterized by the equations:
\begin{equation}\label{EL}
\tau(\varphi)+ K \tau_{\sigma_2}(\varphi)=0,
\end{equation}
where 
\begin{itemize}
\item $\tau(\varphi)=\tr \nabla \dif \varphi$ is the tension field of $\varphi$,
\item $\tau_{\sigma_2}(\varphi)=2[e(\varphi)\tau(\varphi) + 
\dif \varphi(\mathrm{grad}e(\varphi))]-
\tr (\nabla \dif \varphi) \circ \Cc_\varphi 
-\dif \varphi(\di \Cc_\varphi)$ is the \textit{$\sigma_2$--tension field} of $\varphi$, cf. \cite{sacs, slobi, cri}, with $\Cc_\varphi= \dif \varphi^t \circ \dif \varphi \in \mathrm{End}(TM)$ denoting the \textit{Cauchy-Green tensor} of $\varphi$.
\end{itemize}
Obvious solutions for \eqref{EL} are those maps that are both harmonic \textit{and} $\sigma_2$-critical. This is the case for the standard Hopf map $(\Ss^3, can) \to (\Ss^2, can)$, the only exact  solution between (round) spheres known until now. But this situation seems very rare and a heuristic reason for this, given in \cite{slobi}, is that while the prototype for harmonic maps (from a Riemann surface to $\CC$) is provided by a holomorphic/conformal map, the prototype of $\sigma_2$-critical maps is an area-preserving map. In this 2-dimensional context, a map encompasses both conditions if and only if it is homothetic. For a further analysis of \textit{transversally} (to the Reeb foliation) area-preserving maps between 3-dimensional contact manifolds and various stability results, see \cite{slobi}.

\subsection{Sketch of the paper} The main idea of the paper is looking for mappings $\Ss^3 \to \Ss^2$ of higher Hopf invariant, which are either harmonic \textit{or} $\sigma_2$-critical and ask if they produce \textit{full} higher-charge solutions by paying the price of a (bi)conformal change of the domain metric.
In section 2 we point out that a harmonic horizontally conformal submersion becomes $\sigma_{2}$ or $\sigma_{1,2}$-critical if we replace the domain metric with a (bi)conformally related one. Some known examples are revisited.
In section 3 we find (non-conformal) $\sigma_2$-critical maps of arbitrary Hopf invariant on $(\Ss^3, can)$ using a general reduction technique known as $\alpha$-\textit{Hopf construction} \cite{rato, ud, bur, ratto} and we study their stability. We mention that the integrability of the strongly coupled Faddeev-Hopf model on $\Ss^3 \times \RR$ (endowed with a Lorentzian metric of warped product type) has already been proved in \cite{fer}.

\section{Horizontally conformal configurations and related metrics}

Let us recall the following 
\begin{definition} (\cite{ud})
A smooth map $\varphi : (M^m , g) \to (N^n ,h)$ between Riemannian manifolds is a {\em horizontally conformal map} if,
at any point $x\in M$, $\dif \varphi_{x}$ maps the {\em horizontal space} $\HH_{x} = (\ker \dif \varphi_{x})^{\perp}$ conformally onto
$T_{\varphi(x)}N$, i.e. $\dif \varphi_{x}$ is surjective and there exists a number $\lambda (x) \neq 0$ such that
$(\varphi^{*}h)_{x}\Big|_{\HH_{x} \times \HH_{x}} = 
\lambda^{2}(x) g_{x}\Big|_{\HH_{x} \times \HH_{x}}$, or equivalently $(\Cc_{\varphi})_{x} |_{\HH_{x}} = 
\lambda^{2}(x) Id_{TM}|_{\HH_{x}}$.
The function $\lambda$ is the {\em dilation} of $\varphi$; if $\lambda \equiv 1$, then $\varphi$ is a \textit{Riemannian submersion}. If a horizontally conformal map is moreover harmonic, then it is a  \emph{harmonic morphism}. The mean curvatures of the distributions $\HH$ and $\VV = \ker \dif \varphi$ are denoted $\mu^{\HH}$ and $\mu^{\VV}$. 
\end{definition}

In this section we look for horizontally conformal $\sigma_{1,2}$-critical mappings between two Riemannian manifolds. We mention that horizontally conformal condition for complex valued maps has been analyzed in physics literature under the name \textbf{eikonal equation} (see \cite{AGW} and references therein).

\begin{remark}[$\sigma_2$-tension field for horizontally conformal maps, \cite{slobi}]\label{doiunu}
If $\varphi$ is horizontally conformal of dilation $\lambda$, then 
\begin{equation*}
\begin{split}
\tau(\varphi)&= -\dif \varphi \left((n-2)\gr \ln \lambda + (m-n)\mu^{\VV}\right), \\
\tau_{\sigma_2}(\varphi)
&= (n-1)\lambda^2 \left[ \tau(\varphi) + 
2\dif \varphi(\mathrm{grad} \ln \lambda)\right]
=\frac{n-1}{n}\tau_{4}(\varphi),
\end{split}
\end{equation*}
where $\tau_{4}(\varphi)$ is the Euler-Lagrange operator for the 4-energy,
$(1/4) \int_M \abs{\dif \varphi}^4 \nu _g$. In particular, a submersive harmonic morphism is $\sigma_{1,2}$-critical if and only if it is horizontally homothetic (with minimal fibres).
\end{remark}

So the harmonicity and $\sigma_2$-criticality of a horizontally conformal map are related as follows.
\begin{lemma}
Let $\varphi : (M^m , g) \to (N^n ,h)$ with $m \neq 2$ be a horizontally conformal map of dilation $\lambda$. Then $\varphi$ is $\sigma_{1,2}$-critical if and only if it is harmonic with respect to the conformally related metric $\widetilde g$ on $M$, given by
\begin{equation}
\widetilde g = 
\left[1 + K (n-1) \lambda^2 \right]^{\frac{2}{m-2}} \cdot g
\end{equation}
In particular, $\varphi$ is $\sigma_{2}$-critical if and only if it is harmonic with respect to the conformally related metric 
$\widetilde g = \lambda^{\frac{4}{m-2}} \cdot g$.
\end{lemma}

\begin{proof}
Under an arbitrary conformal change of metric $\tilde g = a^2 \cdot g$, the tension field of a map becomes:
$$
\widetilde \tau (\varphi)= \frac{1}{a^2}\left\{ \tau (\varphi) 
+ \dif \varphi (\gr \ln a^{m-2}) \right\}
$$
But, according to the above remark we also have:
$$
\tau_{\sigma_{1,2}} (\varphi)= \left[1 + K (n-1) \lambda^2 \right] \left\{ \tau (\varphi) 
+ \dif \varphi (\gr \ln \left[1 + K (n-1) \lambda^2 \right]) \right\}
$$
\end{proof}

Now let us recall another important type of related metrics.
\begin{definition} (\cite{ud})
Let $(M^m, g)$ be a Riemannian manifold endowed with a distribution $\VV$ of codimension $n$. Denote $\HH = \VV^{\perp}$. Two metrics are \textit{biconformally related with respect to} $\VV$ if it exists a smooth function $\rho: M \to (0, \infty)$ such that: 
\begin{equation}\label{change}
g_{\rho} = \rho^{-2}g^{\mathcal{H}} + \rho^{\frac{2n-4}{m-n}}g^{\mathcal{V}}.
\end{equation}
\end{definition}
The harmonicity of almost submersive maps is invariant under biconformal changes of metric \eqref{change} with respect to $\VV = \Ker \dif \varphi$, cf. \cite{slub, pan}. In particular, for any submersive harmonic morphism $\varphi : (M^m , g) \to (N^n ,h)$ with dilation $\lambda$ and $m > n$, if we take on $M$ the biconformally related metric $g_{\frac{1}{\lambda}}$, then it becomes a Riemannian submersion with minimal fibres (and in particular, $\sigma_2$-critical).

Therefore we got two ways to obtain $\sigma_{1,2}$-critical maps from harmonic morphisms, that we now resume in the following

\begin{pr}\label{key}
Let $\varphi: (M^m,g) \to (N^n,h)$ be a submersive harmonic morphism with $m > n$ and dilation $\lambda$. Then:

\medskip
\noindent $(i.)$ $\varphi$ is $\sigma_{1,2}$-critical with respect to the biconformally related metric $g_{\frac{1}{\lambda}}$ on $M$;

\medskip
\noindent $(ii.)$ $\varphi$ is $\sigma_{1,2}$-critical with respect to the conformally related metric $\widetilde g = b^2 \cdot g$ on $M$ if and only if
\begin{equation}
\gr^{\HH} \left[b^{m-4}(b^2 + K (n-1)\lambda^2)\right] =0.
\end{equation}
In particular, if $m \neq 4$, then $\varphi$ is $\sigma_{2}$-critical with respect to the conformally related metric $\widetilde g = \lambda^{\frac{4}{4-m}} \cdot g$.
\end{pr}

\begin{remark}
\begin{enumerate}
\item[(a)] If $n=2$, then biconformally related metric needed above has a simpler form: $g_{\frac{1}{\lambda}}=\lambda^2 g^{\HH} + g^{\VV}$. 


\item[(b)] Using the $\alpha$-Hopf construction \cite{ratto}, for each pair of positive integers $k$, $\ell$, one can construct a smooth harmonic morphism $\varphi_{k, \ell}: (\Ss^3, e^{2\gamma} \cdot can) \to (\Ss^2, can)$ with Hopf invariant $k\ell$, cf. \cite[Example 13.5.3]{ud} (some details will be also given in the next section). So, by applying Proposition \ref{key}, we can obtain a $\sigma_{1,2}$-critical (or a $\sigma_{2}$-critical) configuration in \textit{every} nontrivial class of $\pi_3(\Ss^2)=\ZZ$ with respect to a metric (bi)conformally related to the canonical one.

\item[(c)] By composing a semiconformal map from $\Ss^4$ to $\Ss^3$ (used in \cite{rato}) with the above mentioned map $\varphi_{k, \ell}$, Burel \cite {bur} has obtained a family of non-constant harmonic morphisms $\Phi_{k, \ell}: (\Ss^4, g_{k, \ell}) \to (\Ss^2, can)$ which represents the (non)trivial class of $\pi_4(\Ss^2)=\ZZ_2$ whenever $k \ell$ is even (respectively odd). In this case too, $g_{k, \ell}$ is in the conformal class of the canonical metric (on $\Ss^4$). 

Again applying Proposition \ref{key}, we can obtain a $\sigma_{1,2}$-critical configuration in the nontrivial class of $\pi_4(\Ss^2)=\ZZ_2$ with respect to a metric (bi)conformally related to the canonical one. Indeed we have only to choose a suitable function $\vartheta$ constant along the horizontal curves and to take $\widetilde g_{k, \ell} = (\vartheta - K \lambda^2) \cdot g_{k, \ell}$.

On the other hand, to obtain a $\sigma_2$-critical point $(\Ss^4, e^{\nu} \cdot can) \to (\Ss^2, can)$ (i.e. an instanton for the strong coupling limit of the Faddeev-Hopf model on Minkowski space) is no more possible with the same procedure, due to conformal invariance in 4 dimensions. A $\sigma_2$-critical map defined on the same pattern as in \cite{bur} may still exist, but it might be not  horizontally conformal.
\end{enumerate}
\end{remark}

\section{Non-conformal higher-charge configurations for the strongly coupled model}  

In \cite{war} Ward has proposed the investigation of the following maps 
\begin{equation}\label{ward}
\Psi_{k, \ell}: \Ss^{3}_{R} \to \CC P^1, \qquad
(z_0, \ z_1) \mapsto 
\left[\frac{z_{0}^{k}}{\abs{z_0}^{k-1}}, \ \frac{z_{1}^{\ell}}{\abs{z_1}^{\ell - 1}} \right], \qquad k, \ell \in \NN^*
\end{equation}
as higher-charge configurations for the Faddeev-Hopf model. He estimated their energy and then compared it to a conjectured topological lower bound.

It is easy to see that $\Psi_{k, \ell}$ are particular cases (via the composition with a version of stereographic projection) of the $\alpha$-\textit{Hopf construction} (applied to $F: \Ss^1 \times \Ss^1 \to \Ss^1$, $F(z, w)=z^k w^{\ell}$) that provides us $\varphi^{\alpha}_{k, \ell}: \Ss^{3}_{R} \to \Ss^2$ defined by:
\begin{equation}\label{aho}
\varphi^{\alpha}_{k, \ell}(R\cos s \cdot e^{\ii x_1}, \ R\sin s \cdot e^{\ii x_2}) = 
\left( \cos \alpha (s), \  \sin \alpha (s) \cdot e^{\ii (kx_1+\ell x_2)}) \right),
\end{equation}
where $k, \ell \in \ZZ^*$ and $\alpha:[0, \pi /2] \to [0, \pi]$ satisfies the boundary conditions $\alpha(0)=0$, $\alpha(\pi /2)=\pi$. When $(k, \ell)=(\mp 1, 1)$ and $\alpha(s) = 2s$, this construction provides us the (conjugate) Hopf fibration. 

The maps $\varphi^{\alpha}_{k, \ell}$ are \textit{equivariant} with respect to some \textit{isoparametric} functions (projections in the argument $s$) and their Hopf invariant is $Q(\varphi^{\alpha}_{k, \ell})=k\ell$ (for more details see \cite{ratto}). They have been considered in many places as the \textit{toroidal ansatz}, see e.g. \cite{AGW, fer, glad, hiet, meis}. 

\medskip 

Let us work out explicitly the condition of being $\sigma_2$-critical for $\varphi^{\alpha}_{k, \ell}$. 

Consider the open subset of the sphere $\Ss^{3}_{R}$ parametrized by
$$
\{p=\left(\cos s \cdot e^{\ii x_1}, \ \sin s \cdot e^{\ii x_2} \right) \ \vert \ (x_1, x_2, s) \in (0, 2\pi)^2 \times (0, \pi / 2) \}.
$$
The (standard) Riemannian metric of $\Ss^{3}_{R}$ is
$g_p=R^2 \left(\cos^{2}s \, \dif x_1^2 + \sin^{2}s \, \dif x_2^2 + \dif s^2 \right)$.

We can immediately construct the orthonormal base for $T_{p}\Ss^{3}_{R}$:
$$
f_1=\frac{1}{R \cos s}\frac{\partial}{\partial x_1}; \quad
f_2=\frac{1}{R \sin s}\frac{\partial}{\partial x_2}; \quad
f_3=\frac{1}{R}\frac{\partial}{\partial s}
$$
Analogously, if we consider 
$\{x=\left(\cos t , \ \sin t \cdot e^{\ii u})\right) \  \vert \  (t, u) \in  (0, \pi) \times (0, 2\pi) \}$,
an open subset of $\Ss^{2}$, then the standard round metric reads \ $h = \dif t^2 + \sin^{2}t \, \dif u^2$.

The differential of the map $\varphi= \varphi^{\alpha}_{k, \ell}$ operates as follows:
\begin{equation}\label{dif}
\dif \varphi (f_1)=\dfrac{k}{R \cos s}\dfrac{\partial}{\partial u}; \quad 
\dif \varphi (f_2) = \dfrac{\ell}{R \sin s}\dfrac{\partial}{\partial u}; \quad
\dif \varphi (f_3) = \dfrac{\alpha^{\prime}(s)}{R}\dfrac{\partial}{\partial t}.
\end{equation}

As we can easily check, the vertical space $\VV = \Ker \dif \varphi$ is spanned by the unitary vector
$$
E_3= \frac{k\ell}{\sqrt{k^2 \sin^2 s + \ell^2 \cos^2 s}}
\left(\frac{\cos s}{k} f_1 - \frac{\sin s}{\ell} f_2 \right)
$$
and the horizontal distribution $\HH = (\Ker \dif \varphi)^{\perp}$, by the unitary vectors
$$
E_1 = f_3, \qquad E_2 = \frac{k\ell}{\sqrt{k^2 \sin^2 s + \ell^2 \cos^2 s}}
\left(\frac{\sin s}{\ell} f_1 + \frac{\cos s}{k} f_2 \right).
$$

A key observation in what follows is that 
\begin{itemize}
\item $E_1$ is an eigenvector for $\varphi^* h$, corresponding to the eigenvalue 
$$
\lambda_{1}^{2}=
\left[\frac{\alpha^{\prime}(s)}{R}\right]^2 ;
$$

\item $E_2$ is an eigenvector for $\varphi^* h$, corresponding to the eigenvalue 
$$
\lambda_{2}^{2}=
\frac{\sin^2 \alpha(s)}{R^2} \cdot \frac{k^2 \sin^2 s + \ell^2 \cos^2 s}{\sin^2 s \cos^2 s}.
$$
Obviously, $E_3$ is an eigenvector too,
corresponding to the zero eigenvalue.
\end{itemize}

\begin{remark}
\begin{enumerate}
\item[(a)] (\textit{Horizontally conformal maps}.) Clearly $\varphi: (\Ss^{3}, can) \to (\Ss^2, can)$ is horizontally conformal provided that $\lambda_{1}^{2} = \lambda_{2}^{2}$. This is a (first order) ODE in $\alpha$:
\begin{equation}
\alpha^{\prime} = \pm \sin \alpha \sqrt{\dfrac{k^2}{\cos^2 s} + \dfrac{\ell^2}{\sin^2 s}}.
\end{equation}
It has a solution for all $k$ and $\ell$, explicitly given in \cite[Example 13.5.3]{ud}. Then taking $\alpha_0$ a solution and performing an appropriate conformal change of metric, we obtain a harmonic morphism $\varphi_{k, \ell}^{\alpha_0}: (\Ss^{3}, e^{2\gamma}can) \to (\Ss^2, can)$.

\item[(b)] (\textit{Harmonic maps}.) For a submersion $\varphi: (M^3,g)\to (N^2,h)$, the equations of harmonicity can be translated in terms of eigenvalues of $\varphi^* h$ as follows, cf. \cite{brd, slub}
\begin{equation}\label{hal}
\left\{
\begin{array}{ccc}
\frac{1}{2}E_1 (\lambda_{2}^{2} - \lambda_{1}^{2}) &=&
(\lambda_{2}^{2} - \lambda_{1}^{2})g(\nabla_{E_2}E_2, E_1) -
\lambda_{1}^{2}g(\mu^{\VV}, E_1)\\[3mm]
\frac{1}{2}E_2 (\lambda_{1}^{2} - \lambda_{2}^{2}) &=&
(\lambda_{1}^{2} - \lambda_{2}^{2})g(\nabla_{E_1}E_1, E_2) -
\lambda_{2}^{2}g(\mu^{\VV}, E_2)
\end{array}
\right.
\end{equation}

For our map given by \eqref{aho}, the second equation is satisfied trivially and the first one leads to a second order ODE in $\alpha$, cf. also \cite{ratto}:
\begin{equation}\label{odeha}
\alpha^{\prime \prime} + (\cot s - \tan s) \alpha^{\prime} - 
\left(\frac{k^2}{\cos^2 s} + \frac{\ell^2}{\sin^2 s}\right)\sin \alpha \cos \alpha = 0.
\end{equation}
According to \cite[Theorem (3.13)]{ratto} and \cite{smi}, \eqref{odeha} has a solution if and only if $\ell = \pm k$; moreover, it is explicitely given by $\alpha(s)=2\arctan\left(C \tan^{k}s\right)$, $C>0$. The corresponding harmonic map $\varphi^{\alpha}_{k, \ell}$ is the standard Hopf map followed by a weakly conformal map of degree $k$.
\end{enumerate}
\end{remark}

We will apply a strategy analogous to the harmonic case described above. Firstly, we need a general result, whose proof is to be find in \cite{slobi}:

\begin{lemma} \label{le}
Let $\varphi: (M^m,g)\to (N^2,h)$ be a submersion. Then $\varphi$ is $\sigma_2$-critical if and only if the following equation is satisfied:
\begin{equation}\label{fh}
\gr^{\HH}(\ln \lambda_{1}\lambda_{2}) - (m-2)\mu^{\VV}=0.
\end{equation}
Moreover, $\varphi$ remains $\sigma_2$-critical when we replace $g$ with $\ov{g}=\sigma^{-2}g^{\HH}+\rho^{-2}g^{\VV}$
if and only if 
$\gr^{\HH}(\sigma^2 \rho^{2-m})=0$, where $\sigma$ and $\rho$ are functions on $M$.
\end{lemma}

We can directly check that, if $\varphi$ is horizontally conformal, i.e. $\lambda_{1}^{2} = \lambda_{2}^{2}$, then \eqref{fh} is equivalent to the 4-harmonicity equation and that \eqref{fh} is equivalent to:
\begin{equation}\label{sil}
\left\{
\begin{array}{ccc}
\frac{1}{2}\lambda_{1}^{2}E_1 (\lambda_{2}^{2}) +
\frac{1}{2}\lambda_{2}^{2}E_1 (\lambda_{1}^{2}) - 
(m-2)\lambda_{1}^{2}\lambda_{2}^{2}g(\mu^{\VV}, E_1) &=& 0\\[3mm]
\frac{1}{2}\lambda_{1}^{2}E_2 (\lambda_{2}^{2}) +
\frac{1}{2}\lambda_{2}^{2}E_2 (\lambda_{1}^{2}) - 
(m-2)\lambda_{1}^{2}\lambda_{2}^{2}g(\mu^{\VV}, E_2) &=& 0
\end{array}
\right.
\end{equation}
For our map given by \eqref{aho}, the second equation in \eqref{sil} is  trivially satisfied and the first one leads to the following second order ODE in $\alpha$:
\begin{equation}\label{odesi2}
\alpha^{\prime} \sin \alpha \left\{ 
[\alpha^{\prime \prime} \sin \alpha + (\alpha^{\prime})^2 \cos \alpha] \cdot 
\left(\frac{k^2}{\cos^2 s} + \frac{\ell^2}{\sin^2 s}\right)
+ \alpha^{\prime} \sin \alpha \cdot 
\left(\frac{k^2}{\sin s \cos^3 s} - \frac{\ell^2}{\cos s \sin^3 s}\right)\right\} = 0.
\end{equation}

Contrary to the harmonic case, this equation always has a (unique) solution, for all $\ell$, $k$ which satisfies boundary conditions $\alpha(0)=0$, $\alpha(\pi /2)=\pi$:

\begin{equation}\label{alfsig}
\alpha(s) = \left\{
\begin{array}{ccc}
\arccos \left(1 - 2 \dfrac{\ln \left(\frac{k^2}{\ell^2}\sin^2 s + \cos^2 s\right)}{\ln \frac{k^2}{\ell^2}}\right)&,&\text{if} \ \abs{k} > \abs{\ell}\\[3mm] 
2s&,&\text{if}  \ \abs{k} = \abs{\ell}\\
\end{array}
\right.
\end{equation}

Therefore, with $\alpha$ given above, we have obtained a   $\sigma_2$-critical map $\varphi_{k, \ell}^{\alpha}$ in every nontrivial homotopy class of $\pi_3(\Ss^2)$. Moreover, they are local minima among equivariant maps of the same type. 

\begin{pr}
The equation \eqref{odesi2} is the Euler-Lagrange equation for the reduced $\sigma_2$-energy functional:
\begin{equation}\label{reden}
\varepsilon_{\sigma_{2}}(\alpha) =
\frac{2\pi^2}{R} \int_{0}^{\frac{\pi}{2}} 
\left(k^2 \tan s + \ell^2 \cot s \right)(\alpha^{\prime})^2 \sin^2 \alpha \, \dif s.
\end{equation}
The solutions \eqref{alfsig} are stable critical points for the energy functional $\varepsilon_{\sigma_{2}}$.
\end{pr}

\begin{proof}
Note that $\mathcal{E}_{\sigma_{2}}(\varphi_{k, \ell}^{\alpha})=\varepsilon_{\sigma_{2}}(\alpha)$. Consider a fixed endpoints variation $\{\alpha_t\}$ of $\alpha$. To prove the result, we only have to follow a direct computation using integration by parts. For the second variation we get:
\begin{equation*}
\dfrac{\dif^2}{\dif t^2} \Bigl\lvert _{0}\varepsilon_{\sigma_{2}}(\alpha_t) =
\frac{4\pi^2}{R} \int_{0}^{\frac{\pi}{2}} 
\left(\dfrac{\dif \alpha_t}{\dif t} \Bigl\lvert _{0}\right)^2 
(\alpha^{\prime})^2 \left(k^2 \tan s + \ell^2 \cot s \right) \, \dif s \geq 0.
\end{equation*}
\end{proof}

\begin{remark}
\begin{enumerate}
\item[(a)]
The $\sigma_2$-critical maps $\varphi_{k, \ell}^{\alpha}$ with $\alpha$ given by \eqref{alfsig} are also harmonic (so critical points for the \textit{full} energy) \textit{only} if $\abs{k} = \abs{\ell}=1$, that corresponds to the (conjugate) Hopf fibration. 

\noindent Recall that in \cite{sve, svee} it has been proved that Hopf map is a \textit{stable} critical point for $\mathcal{E}_{\sigma_{1,2}}$ (if $K \geq 1$) and an absolute minimizer for $\mathcal{E}_{\sigma_{2}}$ (the equivalent quartic energy term used in \cite{sve, svee} is $\int_M \norm{\varphi^*\Omega}^2 \nu_g$).

\noindent For further discussions about critical configurations for the Faddeev-Hopf model on $\Ss^3$ inside the same ansatz, see also \cite{AGW}.

\item[(b)] The $\sigma_2$-energy of critical maps obtained from \eqref{alfsig} is:
\begin{equation}\label{energy}
\mathcal{E}_{\sigma_2}(\varphi_{k, \ell}^{\alpha})=
\frac{16\pi^2}{R} \cdot \frac{k^2-\ell^2}{\ln \left(k^2 / \ell^2 \right)} \ \ (\abs{k} > \abs{\ell}), \quad \mathcal{E}_{\sigma_2}(\varphi_{k, k}^{\alpha})=
\frac{16\pi^2 k^2}{R} \ \ (\abs{k} = \abs{\ell})
\end{equation}
where $Q=k\ell$ is the Hopf charge of the solution (compare with \cite[(29)]{fer}). In particular, the $\sigma_2$-energy of the Hopf map is $\frac{16 \pi^2}{R}$. 

\item[(c)] The $\sigma_2$-critical maps given by \eqref{alfsig} becomes $\sigma_{1,2}$-critical with respect to an appropriately perturbed domain metric $\ov g = \sigma^{-2}g^{\mathcal{H}} + \sigma^{-4}g^{\mathcal{V}}$. Indeed, being $\sigma_2$-critical is invariant under these changes of metric, cf. Lemma \ref{le}, while the tension field becomes $\ov \tau(\varphi_{k, \ell}^{\alpha})=\sigma^2\left[\tau(\varphi_{k, \ell}^{\alpha})+ \dif \varphi_{k, \ell}^{\alpha} (\gr \ln \sigma^{-2})\right]$. But, at least locally, it is possible to find $\sigma$ such that $\ov \tau(\varphi_{k, \ell}^{\alpha})=0$, i.e. $\varphi_{k, \ell}^{\alpha}$ is also harmonic.
\end{enumerate}
\end{remark}

Recall the topological lower bound found in \cite{svee}: $\mathcal{E}_{\sigma_2}(\varphi) \geq 16\pi^2 Q(\varphi)$. In the family of solutions $\varphi_{k, \ell}^{\alpha}$ this bound is attained if and only if $k=\ell$. Therefore:
\begin{pr}
The solution $\varphi_{k, k}^{\alpha=2s}$ is an absolute $\sigma_2$-minima in its homotopy class.
\end{pr}
Notice that, denoting by $\Omega$ the area-form on $\Ss^2$ and by $\eta$ the standard contact form on $\Ss^3$ (see \cite{slobi} for details), we have $(\varphi^{\alpha}_{k, k})^*\Omega=k \, \dif \eta$, i.e. $\varphi^{\alpha}_{k, k}$ is \textit{transversally area-preserving} up to rescale, as expected for heuristic reasons presented in the introduction.

\label{lastpage}
\end{document}